\documentclass[twoside,11pt]{article}

\usepackage{array}
\usepackage[caption=false,font=normalsize,labelfont=sf,textfont=sf]{subfig}
\usepackage{caption}
\usepackage{float}
\usepackage{textcomp}
\usepackage{stfloats}
\usepackage{url}
\usepackage{mathrsfs}
\usepackage{multicol}
\usepackage{verbatim}
\usepackage{longtable}
\usepackage{textcomp}
\usepackage{bm}
\usepackage{amsmath}
\usepackage{amsthm}
\usepackage{amssymb}
\usepackage{braket}
\usepackage{mathrsfs}
\usepackage{amsfonts}
\usepackage{color}
\usepackage{multirow}
\usepackage{cases}
\usepackage{graphicx}
\usepackage{tabularx}
\usepackage{mathrsfs}
\usepackage{mathdots}
\usepackage{algorithm}
\usepackage{algorithmic}
\usepackage{epstopdf}
\usepackage{setspace}
\newtheorem{thm}{Theorem}[]

\theoremstyle{remark}

 \newcommand\bmx{\bm{\mathcal{X}}}
 \newcommand\bmxt{\mathscr{X}_{[T]}}
 
 \newcommand\bmgt{\mathscr{G}_{[T]}}
 \newcommand\bmg{\bm{\mathcal{G}}}
  \newcommand\bmw{\bm{\mathcal{W}}}
  \newcommand\bme{\bm{\mathcal{E}}}
 
 \newcommand\bmu{\bm{U}}
 \newcommand\upt{^{(T)}}

 \newcommand\uptone{^{(T+1)}}
  \newcommand\upto{^{(T_0)}}
 \newcommand\downtone{_{[T+1]}}
 \newcommand\upaaw{^{\text{(AAW)}}}
 \newcommand\downt{_{[T]}}
 \newcommand\setx{\mathscr{X}}
 \newcommand\setg{\mathscr{G}}
 \newcommand\setu{\mathscr{U}}
 \newcommand\setp{\mathscr{P}}
 \newcommand\setxt{\mathscr{X}\downt}
 \newcommand\setgt{\mathscr{G}\downt}

 \newcommand\setgtone{\mathscr{G}\downtone}
 
\usepackage[preprint]{jmlr2e}

\usepackage{lastpage}

\ShortHeadings{Online Prediction for Streaming Tensor Time Series }{Luan, Sun, Wang, and Zhang}
\firstpageno{1}

\begin{document}

\title{Efficient Online Prediction for High-Dimensional Time Series via Joint Tensor Tucker Decomposition}

\author{\name Zhenting Luan \email zhenting.luan@polyu.edu.hk \\
\name Defeng Sun \email defeng.sun@polyu.edu.hk\\
\addr Department of Applied Mathematics\\
The Hong Kong Polytechnic University\\
Hung Hom, Kowloon, Hong Kong
\AND
\name Haoning Wang \email whn22@mails.tsinghua.edu.cn\\
\name Liping Zhang\thanks{Corresponding author.} \email lipingzhang@tsinghua.edu.cn\\
\addr Department of Applied Mathematics\\
Tsinghua University, Beijing, China
}

\editor{My editor}

\maketitle

\begin{abstract}
Real-time prediction plays a vital role in various control systems, such as traffic congestion control and wireless channel resource allocation. In these scenarios, the predictor usually needs to track the evolution of the latent statistical patterns in the modern high-dimensional streaming time series continuously and quickly, which presents new challenges for traditional prediction methods. This paper is the first to propose a novel online algorithm (TOPA) based on tensor factorization to predict streaming tensor time series. The proposed algorithm TOPA updates the predictor in a low-complexity online manner to adapt to the time-evolving data. Additionally, an automatically adaptive version of the algorithm (TOPA-AAW) is presented to mitigate the negative impact of stale data. Simulation results demonstrate that our proposed methods achieve prediction accuracy similar to that of conventional offline tensor prediction methods, while being much faster than them during long-term online prediction. Therefore, TOPA-AAW is an effective and efficient solution method for the online prediction of streaming tensor time series.
\end{abstract}


\section{Introduction}

Time series is a sequence of discrete data points observed over time, which is generated naturally in various areas such as economics, climate, and traffic \citep{kwon2021slicenstitch,shih2019temporal}. The prediction of time series data is crucial in numerous real-world scenarios such as stock market analysis \citep{stockprediction} and traffic flow \citep{introtrafficprediction}. For example, the prediction of traffic is the core component of intelligent transportation systems, including route planning, traffic control and car dispatching \citep{li2018brief}.
Conventional methods for time series prediction usually depend on autoregression (AR) models \citep{ar}, e.g., autoregressive integrated moving average (ARIMA) model \citep{arima} and vector autoregressive (VAR) model \citep{var}, to handle scalar or vector batch data.

However, with the emergence of large-scale time series data generated in modern production activities, classical prediction methods face challenges posed by modern sensing technologies. The modern time series data are usually super large-scale and detected from numerous sensors. For instance, the freeway traffic Performance Measurement System  of the California Department of Transportation is a sensor network of more than 26000 individual-lane inductive loops installed throughout California freeways, with more than 35000 detectors that send traffic measures to a road control unit per 30 seconds \citep{pems,bayes}. Classical methods are not efficient and reliable for handling such large-scale time series. Moreover, modern time series data are often multi-dimensional with different characteristics shown in different dimensions, such as the channel state information in time, spatial, and frequency domains detected by periodic sounding reference signals \citep{csi_intro}.

With the rapid development of sensing technologies, \textit{streaming time series} data are collected continuously from various types of sensors, such as the traffic flow measurement in the traffic congestion control system \citep{streamingtraffic} and the periodic channel state information (CSI) in wireless channel resource allocation \citep{streamingcsi}. Accurate real-time prediction is beneficial for the system to make timely decisions in such scenarios. However, conventional prediction methods must re-exploit the entire dataset, including the newly observed data, to rebuild the prediction models in each process of predicting the upcoming data, which is unsuitable for high-frequency streaming time series.
For example, in mobile scenarios, wireless channels change rapidly. So we're going to see a rapid mismatch of the external channels that rely on pilot estimation. As a result, the performance of interference suppression algorithms is severely degraded. It is crucial to use the prediction algorithm to obtain the channel of the channel at non-pilot time and improve the difference between the estimated value and the actual channel, improve the system performance \citep{kim2020massive}.
Additionally, due to the evolution of statistical patterns in streaming time series data, predicting the upcoming data over an extended period is improper without prediction model updating.
Therefore, online prediction is essential to help forecast upcoming data and make immediate decisions continuously and rapidly.

Recent advances in tensor research have led to the development of prediction models  \citep{ahn2021accurate,wang2024high} for tensor time series (TTS) data with high dimensions based on tensor factorization such as Tucker decomposition and CP decomposition \citep[see][]{de2000multilinear,view,moar}. These methods explore the joint low-rank structure of all observed data, where the statistical patterns or the periodicity of data are utilized for generating predictors. The tensor factorization strategy  \citep[see][]{matrix_factorization} is a generalization of matrix factorization, which has been shown to be very efficient in large-scale multi-variable time series.  \citet{moar} developed a multi-linear prediction model, MOAR, based on Tucker decomposition, which projected the original tensors into several subspaces and built temporal connections among the core tensors by the traditional autoregressive model. The constrained version of MOAR, called MCAR, regularizes the residual of the joint tensor decomposition to find a more stable solution.  \citet{bhtarima} proposed a prediction algorithm that combined tensor decomposition and ARIMA model for the block-Hankelized TTS data, which exploited the low-rank structure of TTS and captured the intrinsic correlations among it. Both of these approaches demonstrate strong performance for single-time predictions by constructing predictors from scratch with numerous iterative iterations. However, they are not suitable for direct application in streaming TTS online prediction due to the necessity of repeating the process of constructing predictors. \citet{completionprediction} presented a short-term traffic flow prediction approach for predicting streaming traffic data based on dynamic tensor completion (DTC). This approach regarded the data to be predicted as some missing data in the corner of a multi-dimensional tensor. It used DTC to complete the tensor to obtain the missing data for prediction. However, this process depends on the multi-mode periodicity of traffic flow data, and hence is unsuitable for generic TTS data without periodicity.

From the above mentioned, using traditional forecasting methods to deal with such scenarios often faces the dual challenges of computational complexity and multi-dimensional feature mining. A natural question is how to  design an accurate and fast online prediction algorithm for such high-dimensional tensor time series.

In this paper, we design an algorithm--- Tucker-decomposition based
Online Prediction Algorithm ({TOPA})--- for predicting generic streaming TTS data in an online manner, which allows us to quickly update the predictors from the previous model and the latest observations rather than rebuild it from scratch. At each sampling time during the streaming observation, we employ the Tucker decomposition for dimension reduction and joint feature subspace extraction of the streaming TTS, and find the auto-regression pattern in the dimension-reduced TTS (the core tensor series). Then, the prediction results are obtained from the predicted core tensor together with the inverse Tucker decomposition.  We design a regularized optimization problem with a least-squared form to search for the proper Tucker decomposition and regression model. In order to quickly solve the optimization and update the previous predictor, we introduce an online updating scheme for the optimization, of which the initialization is exactly the solution to the previous optimization problem formulated at the last time point. Moreover, for streaming data in which the statistical pattern evolves over a long period, we propose an automatically-adaptive-weight (AAW) strategy for processing past data during online prediction, which can reduce the cumulative error of the predictor from stale data and automatically decrease the weights of very noisy data. The improved version of TOPA is called TOPA-AAW.

In summary, the main contribution of this paper is to propose an online prediction algorithm for streaming TTS based on tensor factorization, which
includes the following specific aspects.
\begin{itemize}
    \item The gradually varied joint tensor factorization structure of streaming TTS is updated by the proposed online updating strategy for tracking the changing underlying dynamic statistical characteristics. Compared with conventional joint tensor factorization methods, this strategy takes much less time owing to only a few alternative iterations with an inherited initialization.
    \item  We design an online updating algorithm TOPA with inherited solutions to quickly update the joint tensor factorization structure once observing new data. The prediction accuracy of this algorithm is close to offline methods \citep[e.g.][]{moar,bhtarima}, while the time complexity is significantly lower. The performance of TOPA is validated in numerical simulations, including synthetic and real-world datasets. We also analyze the convergence of our proposed algorithms.
    \item To reduce the interference from stale data to the latest statistical characteristics modeling, we propose an automatically adaptive regularization strategy TOPA-AAW for the introduced least-squared optimization problem. The TOPA-AAW can gradually decrease the weights of historical and abnormal data.
\end{itemize}

The paper is organized as follows. In Section \ref{section_preliminaries}, we define the online prediction problem and introduce the joint tensor factorization strategy for TTS. In Section \ref{section_ostp}, we propose the TOPA and introduce the online manner for predicting streaming TTS data. In order to address the problem of data staleness, we add automatically adaptive weights for regularizing the joint tensor factorization and propose TOPA-AAW in Section \ref{section_aawostp}. In Section \ref{section_experiment}, we evaluate the performance of our proposed algorithms in different scenarios and compare it with some tensor offline prediction methods and neural-network-based methods. Finally, conclusions are drawn in Section \ref{sectioncon_clusion}.

\section{Preliminaries}\label{section_preliminaries}
In this section, we first give notation and tensor operators used in this paper. We next to define the online prediction problem and introduce the joint tensor factorization strategy for TTS.

\subsection{Notation and Tensor Operators} We employ bold italic lower-case letters, bold italic capital letters, and bold  calligraphic letters to denote vectors, matrices, and tensors, respectively (e.g., $\bm{x},\bm{X}$ and $\bmx$).  For a positive integer $n$,  denote the set  $[n]=\{1,2,\ldots,n\}$. We use uppercase letters with script typestyle to denote the set of data (e.g., $\mathscr{G}$). Furthermore, we represent the time series with given length $T$ by uppercase letters with script typestyle equipped with a subscript $[T]$ (e.g., $\mathscr{G}_{[T]}$).  Denote $\bm{I}_k$ as the $k\times k$ identity matrix. We use $(\cdot)^H$ to denote the conjugate transpose of a complex matrix. The set of unitary matrices with given scale $I\times R (I>R)$ is denoted as $$\mathbb{U}^{I\times R}=\{\bm{U}\in\mathbb{C}^{I\times R}|\bm{U}^H\bm{U}=\bm{I}_{R}\}.$$
The Frobenius norm of a tensor $\bmx$ is defined as $\lVert\bmx\rVert_F=\sqrt{\Braket{\bmx,\bmx}}$, where $\Braket{\cdot,\cdot}$ is the inner product. Moreover, the Frobenius norm of a tuple of tensor $\mathcal{X}=(\bmx_1,\ldots,\bmx_n)$ is defined as $\lVert\mathcal{X}\rVert_F=\sqrt{\sum_{i=1}^n\lVert\bmx_i\rVert_F^2}$.
Each dimension of a tensor is called a \textit{mode}. The \textit{$m$-mode product $\times_m$} between a tensor $\bmx\in \mathbb{C}^{I_1\times\cdots \times I_M }$ and a matrix $\bmu_m\in \mathbb{C}^{J_m\times I_m}$ is defined as
\begin{equation*}\label{modeproduct}
(\bmx\times_m \bmu_m)_{i_1\ldots i_{m-1}ji_{m+1}\ldots i_{M}}=\sum_{i_m=1}^{I_m}\bmx_{i_1\ldots i_m\ldots i_M}\bmu_{ji_m}
\end{equation*}
for $\forall j\in [J_m], \forall i_l\in [I_l], \forall l\in [M]\backslash\{m\}.$ The \textit{$m$-mode unfolding matrix} of a tensor is recorded as $\bmx_{(k)}\in \mathbb{C}^{I_m\times(\prod_{l\neq m}I_l)}$. Denote  $R(\bmx)=(rank(\bmx_{(1)}),\cdots,rank(\bmx_{(M)}))$  as the \textit{Tucker-rank} of an $M$th-order tensor $\bmx$.

\subsection{Problem Formulation}\label{subsection_problemdefinition}
Denote $\bmx_t\in \mathbb{C}^{I_1\times\cdots \times I_M}$ as the observation at sampling time $t$ and $\setxt:=\{\bmx_1,\cdots,\bmx_T\}$ as the TTS until sampling time $T$.
Given $\setxt$, the prediction of ${\bmx}_{T+1}$ is to find a predictor $h_T(\cdot)$ that minimizes the mean squared error (MSE) as
\begin{equation}\label{general_model}
    \min\quad \lVert \bmx_{T+1}-h_T(\setxt) \rVert_F,
\end{equation}
where $\hat{\bmx}_{T+1} := h_T(\setxt)$ is the prediction value of $\bmx_{T+1}$.

In this paper, we consider the online prediction problem \eqref{general_model} for streaming TTS, in which the prediction function $h_T(\cdot)$ changes as streaming observations $\bmx_T$ arrive sequentially. Therefore, the previous predictor should be updated with the latest observation for the next prediction at each sampling time. This dynamic process is referred to as \textit{online prediction}.

For high-dimensional streaming TTS data, the curse of dimensionality makes data processing challenging. To mitigate this challenge and explore the multi-aspect temporal continuity of TTS, we introduce the \textit{joint tensor factorization} strategy in the next subsection, which motivates us to develop low-complexity algorithms for online prediction.

\subsection{Joint Tensor Factorization Strategy}\label{subsection_review}
Tensor factorization is a powerful approach for compressing high-dimensional data \citep{view}. In this subsection, we provide a brief overview of this approach, which has been shown to be effective in predicting non-streaming TTS \citep{moar,wang2024high}.

Owing to the temporal continuity of TTS, the TTS data share similar multi-aspect feature subspaces, which can be extracted by Tucker decomposition \citep[see][]{view}. The Tucker decomposition factorizes a higher-order tensor into a core tensor multiplied by a set of projection matrices along each mode. The core tensor captures the essential information of the original tensor, while the projection matrices represent the feature subspaces. In this paper, we utilize joint Tucker decomposition to compress the original TTS into a lower-dimensional representation while preserving the important features. Specifically, the joint Tucker decomposition finds the $(R_1, \cdots, R_M)$-rank approximation of given TTS with joint projection matrices, which is formulated as:
\begin{equation}\label{mcar_tucker}
    {\bmx}_t\approx\bmg_t\upt\prod_{m=1}^M\times_m\bmu_m\upt,\quad  t\in [T],
\end{equation}
where $\bmg_t\upt\in \mathbb{C}^{R_1\times\ldots \times R_M}$ is the \textit{core tensor} corresponding to $\bmx_t$, and $\bm{U}_m\upt\in\mathbb{U}^{I_m\times R_m}$ is considered as the $m$-th joint feature subspace of $\setxt$.

As discussed in \citet{moar},  the core tensors obtained via joint Tucker decomposition can be viewed as a compact representation of the intrinsic interactions among multiple feature subspaces of TTS. This representation is more suitable for modeling temporal continuity than the original data.   Moreover,  the size of core tensors is much smaller than the original TTS, which makes it computationally more efficient to apply AR models to the core tensor series. For instance, we can utilize AR$(p)$ model to illustrate the temporal correlations among the core tensor series, which is formulated as
\begin{equation}\label{mcar_ar}
    \bmg_t\upt=\sum_{i=1}^p\alpha_i\upt\bmg\upt_{t-i}+\bm{\mathcal{E}}_t\upt, \quad p+1\leq t\leq T,
\end{equation}
where $\{\alpha_i\upt\}$ are AR parameters and $\bm{\mathcal{E}}\upt_t$ is the white Gaussian noise in $\bmx_t$. With \eqref{mcar_ar}, the core tensor $\hat{\bmg}_{T+1}\upt$ of next observation data is predicted by
\begin{equation}\label{mcar_core_predict}
\hat{\bmg}_{T+1}\upt=\sum_{i=1}^p\alpha\upt_i\bmg\upt_{T+1-i}.
\end{equation}
Thus, the prediction of upcoming data is given by extending the temporal continuity of multi-aspect feature subspace to the next time point:
 \begin{equation}\label{mcar_data_predict}
 \hat{\bmx}_{T+1}=\hat{\bmg}_{T+1}\upt\prod_{m=1}^M\times_m\bmu_m\upt.
 \end{equation}

When dealing with streaming TTS,  traditional \textit{offline prediction} methods, such as MCAR in \citet{moar} and BHT-ARIMA in \citet{bhtarima}, suffer from high computational complexity as they require finding a new predictor for \eqref{general_model} after each observation without inheriting any information from previous predictors. This results in a time-consuming process of repeatedly solving the model during the streaming observation. To address this issue, we propose an \textit{online prediction} method TOPA for \eqref{general_model} with the tensor-factorization strategy. The proposed TOPA can quickly find new predictors by updating the previous one with the latest observation, thereby reducing computational complexity.

\section{TOPA for Streaming TTS }\label{section_ostp}

In this section, we introduce a novel two-stage online prediction algorithm for \eqref{general_model} called \textit{Tucker-Decomposition-based Online Prediction Algorithm} (TOPA).
The first stage aims to find an initial predictor using starting TTS. This stage can be seen as a generalization of MCAR in \citet{moar} with a simplified optimization problem.
The second stage is the main focus of TOPA, which continuously updates the predictor and predicts the upcoming data as the streaming observation arrives. We online update the predictor by solving an incremental optimization problem with the previous predictor as initialization. At the end of this section, we analyze the convergence and computational complexity of TOPA.

\subsection{Stage \uppercase\expandafter{\romannumeral1}: Initial Predictor with Starting TTS}\label{subsection_starting}

In order to find a suitable initial predictor before receiving streaming data, the observer needs to observe a starting TTS as the initial input for TOPA. It is assumed that this starting TTS comprises $T_0$ samples and the streaming observation starts from time $T_0+1$.

By treating the starting TTS as a non-streaming TTS, we utilize the joint tensor-factorization strategy introduced in Section \ref{subsection_review} to create the initial predictor.
The joint Tucker decomposition of starting TTS is formulated as \eqref{mcar_tucker} with $T=T_0$, which estimates the joint feature subspaces of TTS as $\setu\upto\triangleq\{\bmu_m\upto\}_{m\in [M]}$ and the core tensor series as $\setg_{[T_0]}\upto\triangleq\{\bmg_t
\upto\}_{t\in [T_0]}$.  To be more generalized, we denote $f_{\setp\upto}(\cdot)$ as any desired AR-type model for $\setg_{[T_0]}\upto$, e.g., VAR model and ARIMA model, which is formulated as
\begin{equation}\label{initial_f}
    {\bmg}_{t}\upto\approx f_{\setp\upto}(\setg\upto_{[t-1]}), \quad t\in [T_0],
\end{equation}
where $\setp\upto$ represents all regression parameters in $f$ and $\setg_{[t]}\upto$ are the first $t$ core tensors of the starting TTS. For example, $\setp\upt=\{\alpha_i\upt\}_i$ in \eqref{mcar_ar}.

\begin{remark}
Notice that the order of AR model, denoted $p(<T_0)$, defines the number of past data that have direct temporal correlations with the current data, i.e.,  $$f_{\setp\upto}(\setg\upto_{[t-1]})=f_{\setp\upto}( {\bmg}_{t-p}\upto,\cdots, {\bmg}_{t-1}\upto).$$ For convenience, we still use $f_{\setp\upto}(\setx_{[t]})$ to represent the model with a little abuse of notation when there is no ambiguity about the order.
\end{remark}

With the tensor factorization and regression, TOPA  predicts $\bmx_{T_0+1}$ by inversely projecting the predicted core tensor with the joint feature subspaces:
\begin{align}
    \label{initial_prediction}   \hat{\bmx}_{T_0+1}&=\hat{\bmg}_{T_0+1}\prod_{m=1}^M\times_m\bmu_m\upto   =f_{\setp\upto}\left(\setg\upto_{[T_0]}\right)\prod_{m=1}^M\times_m\bmu_m\upto\nonumber\\
    &=:h_{T_0}(\setx_{[T_0]}),
\end{align}
where $h_{T_o}(\cdot)$ depends on $\setu\upto, \setg\upto_{[T_0]}$ and $\setp\upto$.
Define
\begin{align*}
F_{T_0}\left(\setg_{[T_0]}\upto,\setp\upto,\setu\upto\right)=&
\sum_{t=p+1}^{T_0}\lVert\bmg_{t}\upto-f_{\setp\upto}(\setg_{[t-1]}\upto)\rVert^2_F\\
&\quad+\varphi\sum_{t=1}^{T_0}\lVert\bmg_t\upto-\bmx_t\prod_{m=1}^M\times_m(\bmu\upto_m)^H\rVert^2_F,
\end{align*}
where the first term is used to minimize the error of core tensor series regression, and the second term is used to regularize the joint Tucker decomposition of the starting TTS with a residual minimization formulation and regularization parameter $\varphi>0$.
By \eqref{initial_prediction}, we can reformulate \eqref{general_model} as the following optimization problem:
 \begin{align}
\mathop{\min}_{\setg_{[T_0]}\upto,\setp\upto,\setu\upto} & F_{T_0}\left(\setg_{[T_0]}\upto,\setp\upto,\setu\upto\right)\nonumber \\
\mbox{s.t.}\quad\quad &  \bmu\upto_m\in\mathbb{U}^{I_m\times R_M}, ~~~\forall m\in [M]. \label{optim_initial}
\end{align}

According to the multi-variable and least-squared structure of problem \eqref{optim_initial}, we propose a proximal alternating minimization algorithm to solve it. In this algorithm, we solve a series of subproblems with a proximal regularized term \citep[see][]{proximal_nonconvex} of \eqref{optim_initial} in an alternating manner, with each subproblem fixing all variables except the one being updated. That is to say, it alternatively updates only one of $\setp\upto, \setu\upto$, and $\setg_{[T_0]}\upto$ at a time. Fortunately, each subproblem with respect to $\setp\upto, \setu\upto$, and $\setg_{[T_0]}\upto$  has closed-form solution. We give the updated derivation process in details.

\textbf{Update regression parameters:} For AR model, the subproblem of updating $\setp\upto=\{\bm{\alpha}\upto=[\alpha_1\upto,\cdots,\alpha_p\upto]\}$ with proximal term and step size $\lambda$ is formulated as
    \begin{equation}\label{update_p}
    \begin{aligned}
        \bm{\alpha}_{k+1}\upto & =\mathop{\arg\min}_{\bm{\alpha}}F_{T_0}\left(\setg_{[T_0],k}\upto,\{\bm{\alpha}\},\setu\upto_k\right) + \frac{\lambda}{2}\lVert\bm{\alpha}-\bm{\alpha}_k\upto\rVert_F^2 \\
        & = \mathop{\arg\min}_{\bm{\alpha}}\sum_{t=p+1}^{T_0}\lVert\bmg_{t,k}\upto-\sum_{i=1}^p\alpha_i\bmg\upto_{t-i,k}\rVert^2_F  + \frac{\lambda}{2}\lVert\bm{\alpha}-\bm{\alpha}_k\upto\rVert_F^2. \\
    \end{aligned}
    \end{equation}
 The closed-form solution to \eqref{update_p} is
\begin{equation}\label{update_p_result}
     \bm{\alpha}_{k+1}\upto = \left(\bm{R}+\frac{\lambda}{2}\bm{I}_p\right)^{-1}\left(\bm{r}+\frac{\lambda}{2}\bm{\alpha}_k\upto\right),
\end{equation}
where
\begin{equation}\label{R_matrix_formula}
    \bm{R}_{i,j}=\sum_{t=p+1}^{T_0}\Braket{\bmg_{t-i,k}\upto,\bmg_{t-j,k}\upto}, \quad 1\leq i,j\leq p,
\end{equation}
and
\begin{equation}\label{r_vector_formula}
    \bm{r}_{i}=\sum_{t=p+1}^{T_0}\Braket{\bmg_{t-i,k}\upto,\bmg_{t,k}\upto}, \quad 1\leq i\leq p,
\end{equation}

\begin{remark}
    When $T_0$ is large enough, the entries of $\bm{R}$ and $\bm{r}$ can be seemed as approximations of the autocorrelation function of $\setg_{[T_0],k}\upto$:
$$r(j-i)=\mathbb{E}\Braket{\bmg_t,\bmg_{t+j-i}}.$$
Then, we have
$$\bm{R}_{i,j}\approx r(i-j), \quad \bm{r}_{i}\approx r(I).$$
With these approximations, when $\lambda=0$, \eqref{update_p_result} is the solution to the known Yule-Walker equation \citep{ywequation}.
\end{remark}

\textbf{Update joint projection matrices }$\setu\uptone$: For~$m=1$~to~$M$, the subproblem of updating $\bmu\upto_m$ with proximal term and step size $\lambda$ is formulated as
\begin{align}\label{update_u}
\bmu\upto_{m,k+1}
=&\mathop{\arg\min}_{{\bmu}\in\mathbb{U}^{I_m\times R_m}}F_{T_0}\Big(\setg_{[T_0],k}\upto,\setp\upto_{k+1},
\{{\bmu}_{i,k+1}\upto\}_{i=1}^{m-1}\cup\{\bmu\}\cup\{\bmu_{j,k}\upto\}_{j=m+1}^M\Big) \nonumber\\
& \qquad\qquad\qquad+ \frac{\lambda}{2}\lVert \bmu-\bmu_{m,k}\upto \rVert_F^2 \nonumber\\
=&\mathop{\arg\min}_{{\bmu}\in\mathbb{U}^{I_m\times R_m}}\varphi\sum_{t=1}^{T_0}\lVert\bmg_{t,k}\upto-\bm{\mathcal{H}}_t\times_m\bmu^H\rVert^2_F  + \frac{\lambda}{2}\lVert \bmu-\bmu_{m,k}\upto \rVert_F^2 \nonumber\\
=&\mathop{\arg\max}_{{\bmu}\in\mathbb{U}^{I_m\times R_m}} \Re \text{trace}\left(\bmu^H\left(\sum\limits_{t=1}^{T_0}(\bm{\mathcal{H}}_t)_{(m)}(\bmg_{t,k}\upto)_{(m)}^H+\frac{\lambda}{2\varphi}\bmu_{m,k}\upto\right)\right),
\end{align}
where the last equality holds by utilizing the orthogonal constraints $\bmu^H\bmu=\bm{I}_{R_m}$, the symbol $\Re \text{trace}(\bmx)$ denotes the trace of the real part of $\bmx$,  and \begin{equation}\label{HM}
\bm{\mathcal{H}}_t=\bmx_{t}\prod_{i<m}\times_i({\bmu}_{i,k+1}\upto)^H\prod_{j>m}\times_j(\bmu_{j,k}\upto)^H.
\end{equation}

With orthogonal constraints, the closed-form solution of \eqref{update_u} is
\begin{equation}\label{update_u_result}
    \bmu\upto_{m,k+1} = \bm{L}_{m,k+1}\bm{R}_{m,k+1}^H,
\end{equation}
where $\bm{L}_{m,k+1}\in\mathbb{U}^{I_m\times R_m}$ and $\bm{R}_{m,k+1}\in\mathbb{U}^{R_m\times R_m}$ are the left and right singular matrices of \begin{equation}\label{lr}
   \sum\limits_{t=1}^{T_0}(\bm{\mathcal{H}}_t)_{(m)}(\bmg_{t,k}\upto)_{(m)}^H+\frac{\lambda}{2\varphi}\bmu_{m,k}\upto.
\end{equation}

\textbf{Update core tensor series ~}$\setg_{[T_0]}\upto$:  For~$t=1$~to~$T_0$, the subproblem of updating $\bmu\upto_m$ with proximal term and step size $\lambda$ is formulated as
\begin{align}\label{update_g}
{\bmg}_{t,k+1}\upto=\mathop{\arg\min}_{{\bmg}_t}F_{T_0}\Big(
\setg_{[t-1],k+1}\upto&\cup\{{\bmg}_t\}\cup
\{\bmg_{s,k}\upto\}_{s=t+1}^{T+1}
,\setp\upto_{k+1},\setu\upto_{k+1}\Big)\nonumber\\& + \frac{\lambda}{2}\lVert \bmg_t-\bmg_{t,k}\upto \rVert^2_F.
\end{align}
Since \eqref{update_g} is a quadratic optimization problem, when the gradient of the objective function with respect to $\bmg_t$ vanishes, we can obtain the solution to \eqref{update_g}:
\begin{equation}\label{update_g_result1}
\bmg_{t,k+1}\upto=
\dfrac{1}{\varphi+\frac{\lambda}{2}}\left[\varphi\bmx_{t}\prod_{m=1}^M\times_m({\bmu}_{m,k+1}\upto)^H
+\dfrac{\lambda}{2}\bmg_{t,k}\upto\right]\quad \text{if $t\leq p$},
\end{equation}
and
\begin{equation}\label{update_g_result2}
\bmg_{t,k+1}\upto=
\dfrac{1}{1+\varphi+\frac{\lambda}{2}}\Big[f_{\setp\upto_{k+1}}(\setg_{[t-1],k+1}\upto)+
\varphi\bmx_{t}\prod_{m=1}^M\times_m({\bmu}_{m,k+1}\upto)^H+\dfrac{\lambda}{2}\bmg_{t,k}\upto\Big]\quad \text{if $t> p$}.
\end{equation}
As $t$ traverses $[T+1]$, $\setg_{[T_0],k}\upto$ is updated to $\setg_{[T_0],k+1}\upto:=\{{\bmg}_{t,k+1}
\upto\}_{t\in [T]}$.

According to \eqref{HM} and \eqref{lr}, define
\begin{equation}\label{xmk}
\bmw^{(T_0)}=\sum\limits_{t=1}^{T_0}\left(\bmx_{t}\prod\limits_{i<m}\times_i({\bmu}_{i,k+1}\upto)^H
\prod\limits_{j>m}\times_j(\bmu_{j,k}\upto)^H\right)_{(m)}\times\left(\bmg_{t,k}\upto\right)_{(m)}^H
+\frac{\lambda}{2\varphi}\bmu_{m,k}\upto.
\end{equation}
Then, we summarize this stage of TOPA in  Algorithm \ref{algorithm_init}. In each iteration, we first update the regression model of the core tensor series in terms of current iteration results of core tensors (line $2$ in Algorithm \ref{algorithm_init}), and then update the joint projection matrices and core tensor series in the joint Tucker decomposition structure (line $3\! -\! 6$ and line $7\! -\! 12$, respectively).

\begin{algorithm}
 \caption{TOPA - Stage \uppercase\expandafter{\romannumeral1}}\label{algorithm_init}
 \begin{algorithmic}[1]
 \renewcommand{\algorithmicrequire}{\textbf{Input:}}
 \renewcommand{\algorithmicensure}{\textbf{Output:}}
 \REQUIRE TTS $\setx_{[T_0]}$, and autoregressive model $f$, step size $\lambda>0$, tolerance $\epsilon>0$ and iteration number $k=0$
 \ENSURE Core tensor series $\setg_{[T_0]}\upto$, autoregressive parameters $\setp\upto$ and joint projection matrices $\setu\upto$
     \\\hrulefill
\\  Initialize joint projection matrices  $\setu\upto_0=\{\bmu\upto_{m,0}\}_{m=1}^M$ randomly.
 \\ Initialize core tensor series:
 \\ \quad $\setg_{[T_0],0}\upto=\{\bmg_{t,0}\upto:=\bmx_t\prod_{m=1}^M\times_m(\bmu_{m,0}\upto)^H\}_{t=1}^{T_0}$.

\WHILE{$k\geq 0$ }
 \STATE Compute regression parameters $\setp\upto_{k+1}$ for $f$ with $\setg_{[T_0],k}\upto$ via solving subproblem \eqref{update_p}.
\FOR{$m=1:M$}
\STATE Compute  the left and right singular matrices $\bm{L_{m,k+1}}$ and $\bm{R_{m,k+1}}$ respectively of $\bmw^{(T_0)}$ defined by \eqref{xmk}.
  \STATE Update ${\bmu}\upto_{m,k+1}$ via \eqref{update_u_result}.
  \ENDFOR
  \FOR{$t=1:p$}
  \STATE Update $\bmg_{t,k+1}\upto$ via \eqref{update_g_result1}.
\ENDFOR
  \FOR{$t=p+1:T_0$}
  \STATE Update $\bmg_{t,k+1}\upto$ via \eqref{update_g_result2}.
\ENDFOR
\IF{$\lVert\setg_{[T_0],k+1}\upto- \setg_{[T_0],k}\upto \rVert_F^2 +\lVert  \setu_{k+1}\upto- \setu_k\upto \rVert_F^2 + \lVert\setp_{k+1}\upto- \setp_k\upto \rVert_F^2<\epsilon $}
\STATE Break;
\ENDIF
\STATE $k\leftarrow k+1$
 \ENDWHILE
 \RETURN $\setp\upto\!=\!\setp\upto_k,~~~~\setu\upto\!=\!\setu\upto_k, ~~~~ \setg_{[T_0]}\upto\!=\!\setg_{[T_0],k}\upto$
 \end{algorithmic}
 \end{algorithm}

By executing Algorithm \ref{algorithm_init}, we can build an initial predictor \eqref{initial_prediction} for the starting TTS, which is the basis for the subsequent online prediction stage.

\subsection{Stage \uppercase\expandafter{\romannumeral2}: Online Predictor Updating and Prediction}\label{subsection_online}

In this stage, new streaming TTS data arrive in sequence. After each data sampling, we update the previous predictor online and then predict the upcoming TTS data with the updated predictor.
The online manner relies on the same assumption of temporal continuity of considered TTS as discussed in Section \ref{subsection_review}. In this stage, we repeat \textit{online updating} and \textit{prediction} for online prediction, as illustrated in Figure \ref{fig_flow1}, where $T_0$ in \eqref{update_p}, \eqref{update_u_result} and \eqref{update_g} should be replaced by $T+1$. We list the main two steps as follows.

\begin{figure*}[!h]
    \centering    \includegraphics[width=0.85\textwidth,height=\textheight,keepaspectratio]{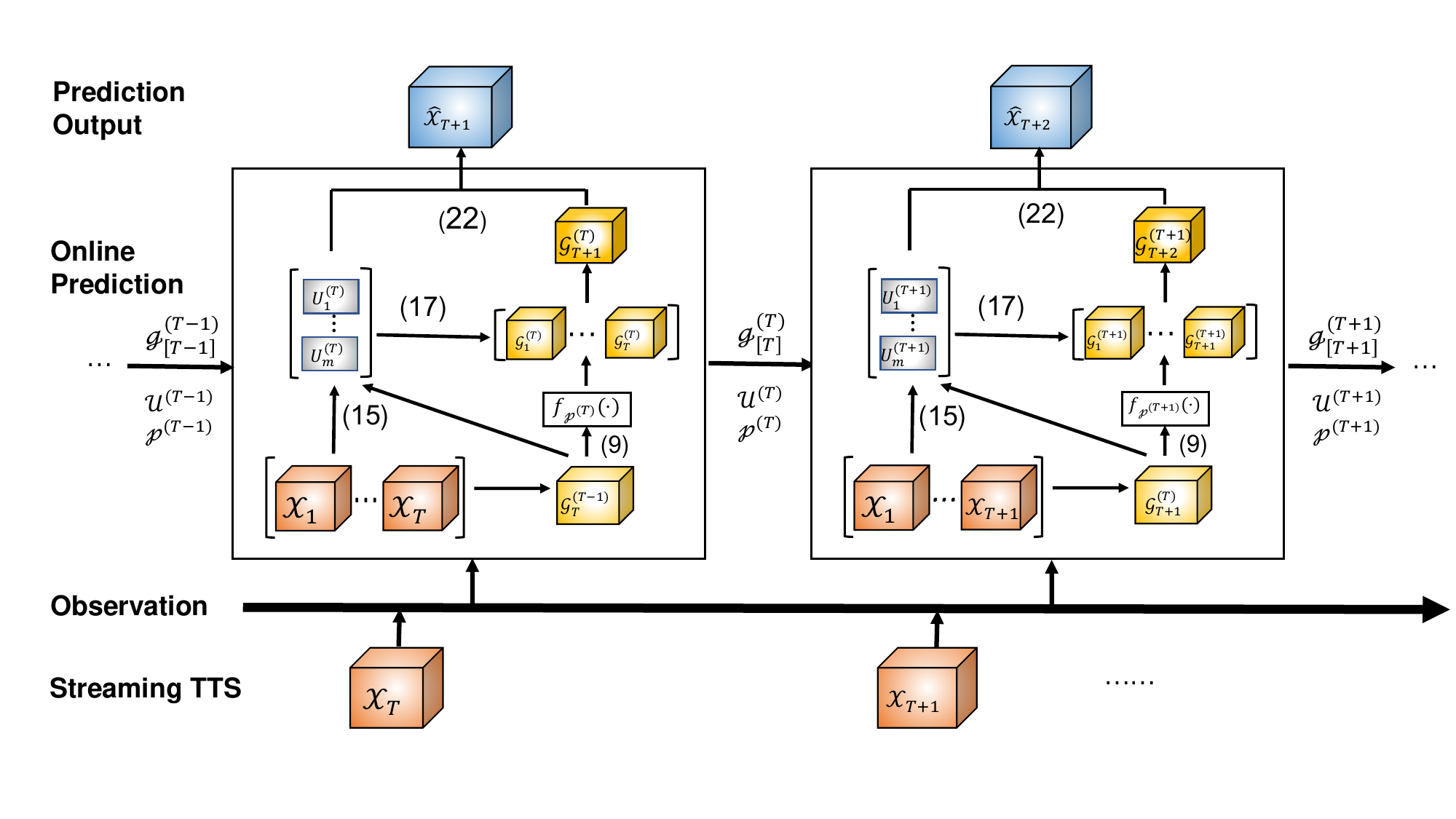}\vspace{-3mm}
    \caption{Stage \uppercase\expandafter{\romannumeral2} of TOPA for Streaming TTS}
    \label{fig_flow1}
\end{figure*}

\underline{\textbf{Online updating step:}} At time $T+1(>T_0)$, new TTS data $\bmx_{T+1}$ arrives. Similar to Stage \uppercase\expandafter{\romannumeral1}, the process of finding new predictor $h_{T+1}(\cdot)$ can be formulated as
 \begin{align}\label{optim_updating}
\mathop{\min}_{\setg\downtone\uptone,\setp\uptone,\atop\setu\uptone} & F_{T+1}\left(\setg\downtone\uptone,\setp\uptone,\setu\uptone\right) \nonumber\\
\mbox{s.t.}\quad \quad &  \bmu\uptone_m\in\mathbb{U}^{I_m\times R_M}, \quad \forall m\in [M],
\end{align}
where
\begin{align*}
F_{T+1}\left(\setg\downtone\uptone,\setp\uptone,\setu\uptone\right)=&
\sum_{t=p+1}^{T+1}\lVert\bmg_{t}\uptone-f_{\setp\uptone}(\setg_{[t-1]}\uptone)\rVert^2_F\nonumber\\
&\quad +\varphi\sum_{t=1}^{T+1}\lVert\bmg_t\uptone-\bmx_t\prod_{m=1}^M\times_m(\bmu\uptone_m)^H\rVert^2_F.
\end{align*}

Under the assumption of temporal continuity of TTS, the joint feature subspaces among TTS data and core tensor series change smoothly as observation continues \citep{moar}. Moreover, compared to $F_{T}(\cdot)$, $F_{T+1}(\cdot)$ contains only two additional terms that relate to $\bmx_{T+1}$. Thus the change in the objective function for the online optimization problem is relatively small. When $h_T(\cdot)$ is accurate enough, we can employ it as the initialization for solving \eqref{optim_updating}. Specifically, we take  $\bmgt\upt$, $\setu\upt$ and $\setp\upt$ as the initialization of (\ref{optim_updating}), and use an alternative updating scheme to solve (\ref{optim_updating}) in closed forms. The details of alternative updating are shown in the online updating step of Algorithm \ref{algorithm_main}. The updates of $\setg\downtone\uptone$, $\setp\uptone$, and $\setu\uptone$ are similar to \eqref{update_p_result}, \eqref{update_u_result}, \eqref{update_g_result1} and \eqref{update_g_result2}. We omit the process of update here.

\underline{\textbf{Online prediction step: }} After online updating the predictor with new observation, TOPA predicts the upcoming tensor data at time $T+2$ by
\begin{align}
    \label{online_prediction}   \hat{\bmx}_{T+2}&=f_{\setp\uptone}\left(\setg\uptone_{[T+1]}\right)\prod_{m=1}^M\times_m\bmu_m\uptone\nonumber\\
    &=:h_{T+1}(\setx_{[T+1]}).
\end{align}

In the practice of our proposed online manner, we only need to run a few iterations to find a predictor accurate enough, owing to the inheritance of the previous solution. In contrast, offline prediction algorithms, such as those discussed in \cite{moar} and \cite{bhtarima}, need to perform multiple iterations with random initialization to solve \eqref{optim_updating}, which is inefficient and repetitive during streaming observations. The simulations in Section \ref{section_experiment} demonstrate that our online updating scheme performs as well as those offline methods, even with just one iteration.

\begin{algorithm}
 \caption{TOPA - Stage \uppercase\expandafter{\romannumeral2}}\label{algorithm_main}
 \begin{algorithmic}[1]
 \renewcommand{\algorithmicrequire}{\textbf{Input:}}
 \renewcommand{\algorithmicensure}{\textbf{Output:}}
 \REQUIRE Streaming TTS $\bmxt$ from $T=T_0$, autoregressive model $f$,  step size $\lambda>0$, and tolerance $\epsilon>0$
 \ENSURE Online prediction results $\{\hat{\bmx}_{t}\}_{t=T_0+2}^\infty$
 \\\hrulefill
  \WHILE{ \textbf{observing new data $\bmx_{T+1}$ at time $T+1$:} }
  \STATE $k\leftarrow 0$
  \STATE\underline{\textbf{step 1: online updating}}  \\
  \STATE  $\bmg_{T+1}\upt:=\frac{f_{\setp\upt}(\setgt\upt)+\varphi\bmx_{T+1}\prod\limits_{m=1}\limits^M\times_m(\bmu_m\upt)^H}{1+\varphi}$ \\
  \STATE $\setp\uptone_k\!=\!\setp\upt,\setu\uptone_k\!=\!\setu\upt, \setg_{[T+1],k}\uptone\!=\!\setgtone\upt$
  \WHILE{$k\geq 0$ }
  \STATE Compute regression parameters $\setp\uptone_{k+1}$ for $f$ with $\setg_{[T+1],k}\uptone$ via solving \eqref{update_p}.
\FOR{$m=1:M$}
\STATE Compute  the left and right singular matrices $\bm{L_{m,k+1}}$ and $\bm{R_{m,k+1}}$ respectively
of $\bmw^{(T+1)}$ defined as \eqref{xmk}.
  \STATE Update ${\bmu}\uptone_{m,k+1}\leftarrow\bm{L}_{m,k+1}\bm{R}_{m,k+1}^H$
  \ENDFOR
  \FOR{$t=1:p$}
  \STATE $\bmg_{t,k+1}\uptone\leftarrow\dfrac{\varphi\bmx_{t}\prod_{m=1}^M\times_m({\bmu}_{m,k+1}\uptone)^H+\frac{\lambda}{2}\bmg_{t,k}\uptone}{\varphi+\frac{\lambda}{2}}.$
\ENDFOR
  \FOR{$t=p+1:T+1$}
  \STATE  $\bmg_{t,k+1}\uptone\leftarrow\dfrac{f_{\setp\uptone_{k+1}}(\setg_{[t-1],k+1}\uptone)+\varphi\bmx_{t}\prod\limits_{m=1}\limits^M\times_m({\bmu}_{m,k+1}\uptone)^H+\frac{\lambda}{2}\bmg_{t,k}\uptone}{1+\varphi+\frac{\lambda}{2}}.$
\ENDFOR
\IF{$\lVert\setg_{[T+1],k+1}\uptone- \setg_{[T+1],k}\uptone \rVert_F^2 +\lVert  \setu_{k+1}\uptone- \setu_k\uptone \rVert_F^2 + \lVert\setp_{k+1}\uptone- \setp_k\uptone \rVert_F^2<\epsilon $}
\STATE Break;
\ENDIF
\STATE $k\leftarrow k+1$ \\
  \ENDWHILE \\
  \STATE $\setp\uptone\!=\!\setp\uptone_k,~~~\setu\uptone\!=\!\setu\uptone_k,~~~ \setgtone\uptone\!=\!\setg_{[T+1],k}\uptone$\\
  \hrulefill \\
  \underline{\textbf{step 2: online prediction}}\\
  \STATE \quad\quad
  $\hat{\bmx}_{T+2}=f_{\setp\uptone}(\setg\downtone\uptone)\prod_{m=1}^M\times_m\bmu_m^{(T+1)}.$
 \RETURN $\hat{\bmx}_{T+2}$
 \STATE $T\leftarrow T+1$\\
 \ENDWHILE
 \end{algorithmic}
 \end{algorithm}

\subsection{Convergence and Complexity Analysis}\label{subsection_complexity}

In this subsection, we analyze the convergence and complexity of TOPA. Since the online updating process of Algorithm \ref{algorithm_main} is similar to Algorithm \ref{algorithm_init}, we only provide  convergence analysis for Algorithm \ref{algorithm_init} here.

Denote the sequence generated by Algorithm \ref{algorithm_init} as $\left\{\mathcal{W}_k\upto=\left(\setg_{[T_0],k}\upto,\setp\upto_k,\setu_k\upto\right)\right\}$. Let $\mathcal{W}_*\upto=\left(\setg_{[T_0],*}\upto,\setp\upto_*,\setu_*\upto\right)$ be any of its accumulation point. Let $N(\bmu\upto_{m,*})$ be the normal cone  of $\mathbb{U}^{I_m\times R_m}$ at $\bmu\upto_{m,*}$, as defined in \citet{rock_variational}. Then, we have the following convergence result.

\begin{thm}\label{thm_convergence}
Assume that the starting TTS and the regression parameter sequence $\{\setp_k\upto\}_k$ generated by Algorithm \ref{algorithm_init} are bounded. Then, for any limit points $\mathcal{W}_*\upto$ of the sequence $\left\{\mathcal{W}_k\upto\right\}$, we have
\begin{equation}\label{gradient_converge_pg}
    \frac{\partial F_{T_0}\left(\mathcal{W}_*\upto\right)}{\partial \setp\upto} = \bm{0}, \quad
    \frac{\partial F_{T_0}\left(\mathcal{W}_*\upto\right)}{\partial \setg\upto_{[T_0]}} = \bm{0},
\end{equation}
and
\begin{equation}\label{gradient_converge_u}
  -\frac{\partial F_{T_0}\left(\mathcal{W}_*\upto\right)}{\partial \bmu_{m}\upto} \in N(\bmu\upto_{m,*}).
\end{equation}
In other words, any limit point of the iteration sequence generated by Algorithm \ref{algorithm_init} is the stationary point of \eqref{optim_initial}.
\end{thm}
\begin{proof}
Denote $\mathcal{W}\upto_k = (\setg_{[T_0],k}\upto, \setp_k\upto, \setu_k\upto)$.  In terms of the alternative manner  of Algorithm \ref{algorithm_init}, we have
\begin{equation}\label{proof_inequality}
\begin{aligned}
          F_{T_0}\left(\setg_{[T_0],k}\upto,\setp_k\upto,\setu\upto_k\right)
    {\geq} & F_{T_0}\left(\setg_{[T_0],k}\upto,\setp_{k+1}\upto,\setu\upto_k\right) + \frac{\lambda}{2}\lVert  \setp_{k+1}\upto- \setp_k\upto \rVert_F^2 \\
    {\geq} & F_{T_0}\left(\setg_{[T_0],k}\upto,\setp_{k+1}\upto,\setu\upto_{k+1}\right) +
    \frac{\lambda}{2}\lVert  \setp_{k+1}\upto- \setp_k\upto \rVert_F^2 \\
    &\qquad\qquad\qquad+\frac{\lambda}{2}\lVert  \setu_{k+1}\upto- \setu_k\upto \rVert_F^2 \\
    {\geq } & F_{T_0}\left(\setg_{[T_0],k+1}\upto,\setp_{k+1}\upto,\setu\upto_{k+1}\right) +  \frac{\lambda}{2}\lVert  \mathcal{W}_{k+1}\upto- \mathcal{W}_k\upto \rVert_F^2,
\end{aligned}
\end{equation}
where the three inequalities are due to the optimization subproblems \eqref{update_p}, \eqref{update_u}, and \eqref{update_g}, respectively. Therefore, via recursion, we have
\begin{equation}\label{proof_inequality_k}
 F_{T_0}\left(\setg_{[T_0],0}\upto,\setp_0\upto,\setu\upto_0\right)
    \geq  F_{T_0}\left(\setg_{[T_0],k}\upto,\setp_{k}\upto,\setu\upto_{k}\right) + \sum\limits_{i=1}^k \frac{\lambda}{2}\lVert  \mathcal{W}_{i}\upto- \mathcal{W}_{i-1}\upto \rVert_F^2.
\end{equation}

For any $k$ and $t\leq T_0$, with the boundedness of $\bmx_{t}$ and ${\bmu}_{m,k}\upto$, we have
\begin{equation}\label{g_bound}
\begin{aligned}
        &\lVert \bmg_{t,k}\upto \rVert_F^2 \leq \lVert \bmg_{t,k}\upto - \bmx_{t}\prod_{m=1}^M\times_m({\bmu}_{m,k}\upto)^H\rVert_F^2 + 2\lVert\bmx_{t}\rVert_F^2 \\
    \leq & 2F_{T_0}\left(\setg_{[T_0],k}\upto,\setp_{k}\upto,\setu\upto_{k}\right) + 2\lVert\bmx_{t}\rVert_F^2 \\
    \leq & 2F_{T_0}\left(\setg_{[T_0],0}\upto,\setp_{0}\upto,\setu\upto_{0}\right) + 2\lVert\bmx_{t}\rVert_F^2.
\end{aligned}
\end{equation}
Hence $\{\setg_{[T_0],k}\upto\}_k$ is bounded. Therefore, $\{\mathcal{W}\upto_k\}$ is bounded and has at least one limit point.

With \eqref{proof_inequality_k}, we have
\begin{equation}\label{series_w}
    \sum\limits_{i=1}^{+\infty}\lVert \mathcal{W}_{i}\upto- \mathcal{W}_{i-1}\upto \rVert_F^2 < \infty,
\end{equation}
which implies
\begin{equation}\label{w_converge}
    \mathcal{W}_{k+1}\upto- \mathcal{W}_{k}\upto \rightarrow \bm{0}.
\end{equation}

For any limit point $\mathcal{W}_*\upto=(\setg_{[T_0],*}\upto,\setp\upto_*,\setu_*\upto)$ of $\{\mathcal{W}_k\upto\}$, there exists a subsequence $\{\mathcal{W}_{k_i}\upto\}_i$ that converges to $\mathcal{W}_*\upto$. With \eqref{w_converge}, we have $\mathcal{W}_{k_i+1}\upto\rightarrow\mathcal{W}_*\upto$.

For any $t$ and $i$, since $\bmg_{t,k_i+1}\upto$ is the solution to \eqref{update_g} with index $k_i$, we have
\begin{equation}\label{zerogradient_g}
\frac{\partial F_{T_0}}{\partial \bmg_{t}\upto}\big(
\setg_{[t-1],k_i+1}\upto\cup\{\bmg_{t,k_i+1}\upto\} \cup
\{\bmg_{s,k_i}\upto\}_{s=t+1}^{T+1}
,\setp\upto_{k_i+1},\setu\upto_{k_i+1}\big) + \lambda\left( \bmg_{t,k_i+1}\upto -  \bmg_{t,k_i}\upto  \right) = \bm{0}.
\end{equation}
Let $i\rightarrow +\infty$ in \eqref{zerogradient_g}, we have
\begin{equation}\label{zerogradient_g_limit}
     \frac{\partial F_{T_0}\left(\mathcal{W}_*\upto\right)}{\partial \bmg\upto_{t}} = \bm{0}.
\end{equation}
Similarly, we have
\begin{equation}\label{zerogradient_p_limit}
     \frac{\partial F_{T_0}\left(\mathcal{W}_*\upto\right)}{\partial \setp\upto} = \bm{0}.
\end{equation}
Then, \eqref{gradient_converge_pg} is proved.

For any $m$ and $i$, since $\bmu_{m,k_i+1}\upto$ is the solution to \eqref{update_u} with index $k_i$, by Theorem 8.15 in \citet{rock_variational}, we have
\begin{equation}\label{zerogradient_u}
\begin{aligned}
   -\frac{\partial F_{T_0}}{\partial \bmu_{m}\upto}&\Big(\setg_{[T_0],k_i}\upto,\setp\upto_{k_i+1},\{{\bmu}_{i,k_i+1}\upto\}_{i=1}^{m-1}\cup\{\bmu_{m,k_i+1}\upto\}\cup\{\bmu_{j,k_i}\upto\}_{j=m+1}^M\Big) \\
   & - \lambda\left(\bmu_{m,k_i+1}\upto-\bmu_{m,k_i}\upto\right) \in N(\bmu_{m,k_i+1}\upto).
\end{aligned}
\end{equation}
Let $i\rightarrow +\infty$, with Proposition 6.6 in \citet{rock_variational} and \eqref{zerogradient_u}, we obtain
\begin{equation}\label{zerogradient_u_limit}
\begin{aligned}
   -\frac{\partial F_{T_0}}{\partial \bmu_{m}\upto}\left(\mathcal{W}_*\upto\right) \in N(\bmu_{m,*}\upto).
\end{aligned}
\end{equation}
Hence, \eqref{gradient_converge_u} holds.
\end{proof}

The computational complexity of TOPA-Stage \uppercase\expandafter{\romannumeral2} is dominated by updating joint projection matrices and core tensor series. For each iteration in online prediction, the total computational complexity is  $O\left(2MT\Bar{I}^M\Bar{R}+M(M-1)T\Bar{R}^M\Bar{I}\right)$, where $\Bar{I}=(\prod_{m\in[M]}I_m)^{1/M}$ and $\Bar{R}=(\prod_{m\in[M]}R_m)^{1/M}$ are the geometric average of the scale of TTS data and core tensors, respectively.

For details, we list the computational complexity of each iteration during online prediction by TOPA in Table \ref{comlexity_ostp}. Here we denote the total order of AR model $f$ is $p$, which is assumed much less than the scale of tensor data.

\begin{table}[!ht]
\centering
    \caption{Computational Complexity of Stage \uppercase\expandafter{\romannumeral2} in TOPA}\label{comlexity_ostp}
\begin{tabular}{cc}
\hline
Step  & Computational Complexity \\ \hline
Line $4$ of TOPA-Stage \uppercase\expandafter{\romannumeral2}   &      $O\left(M\Bar{I}^M\Bar{R}\right)$                    \\
Computing $\setp\uptone_{k+1}$  &     $O\left(p^3\Bar{R}^M\right)$                     \\
\eqref{update_u_result}   &     $O\left(MT((M-1)\Bar{I}^M\Bar{R}+\Bar{R}^M\Bar{I})\right)$                     \\
\eqref{update_g}   &   $O\left(T(M\Bar{I}^M\Bar{R}+p\Bar{R}^M)\right)$                         \\
\hline
Total &    $O\left(M^2T\Bar{I}^M\Bar{R}+MT\Bar{R}^M\Bar{I}\right)$                         \\ \hline
\end{tabular}
\end{table}

\section{TOPA-AAW: TOPA with Automatically Adaptive Weights }\label{section_aawostp}

In real-world scenarios, when streaming TTS data are observed continuously over a relatively long period, the effectiveness of stale data in revealing the latest statistical patterns in streaming TTS gradually diminishes. This leads to an increasing cumulative prediction error, such as wireless CSI \citep{streamingcsi}, in which the intrinsic statistical patterns evolve rapidly as observations continue, and the data becomes outdated quickly. To address this issue, we propose an automatically adaptive weight (AAW) regularization method with a time sliding window $\tau$ for modifying \eqref{optim_updating}. The AAW gradually decreases the weights of stale data to reduce their impacts on prediction accuracy. To further eliminate the interference caused by heavily noisy data, we introduce an automatic factor in AAW to reduce the weights of heavily noisy data.

The optimization problem for online prediction, once $\bmx_{T+1}$ is observed at time $T+1$, is modified as:
 \begin{align}\label{optim_aaw}
\mathop{\min}_{\setp\uptone,\setu\uptone,\atop
\setg\downtone\uptone}& F\upaaw_{T+1}(\setg\downtone\uptone,\setp\uptone,\setu\uptone)\nonumber\\
\mbox{s.t.}\qquad\quad &  \bmu\uptone_m\in\mathbb{U}^{I_m\times R_M}, \quad \forall m\in [M],
\end{align}
where
\begin{equation*}
\begin{aligned}
F\upaaw_{T+1}(\setg\downtone\uptone,\setp\uptone,\setu\uptone)=&
\sum_{t=T-\tau+2}^{T+1}\Big(\lVert\bmg_{t}\uptone-f_{\setp\uptone}(\setg_{[t-1]}\uptone)\rVert^2_F\\
&\qquad+
\varphi\omega_t\uptone\lVert\bmg_t\uptone-\bmx_t\prod_{m=1}^M\times_m(\bmu\uptone_m)^H\rVert^2_F\Big),
\end{aligned}
\end{equation*}
and  $\omega_t\uptone$ is the AAW
\begin{equation}\label{aaw}
    \omega_t\uptone=\begin{cases}
(1-\alpha^{t-(T-\tau+1)})\max\{\beta,1-\epsilon_t\uptone\},  &T-\tau+2\leq t\leq T,\\
1,&t=T+1
\end{cases}
\end{equation}
with Tucker-decomposition residual error
\begin{equation}
     \epsilon_t\uptone=\frac{\lVert\bmx_t-{\bmg}_t\uptone\prod_{m=1}^M\times_m{\bmu}_m\uptone\rVert_F^2}{\lVert\bmx_t\rVert_F^2}.
\end{equation}
Here $\alpha\in (0,1)$ is the damping parameter for reducing the accumulative prediction error from stale data. The `max' factor is inversely proportional to the residual error $\epsilon_t\uptone$ of joint tensor decomposition for streaming TTS in the last prediction, while $\beta\in (0,1)$ is the minimum residual factor.

\begin{remark}\label{remark_beta}
The parameter $\beta$ is configured as a threshold for checking the deviation of the joint Tucker decomposition for each data in TTS, which can help prevent the data with significant decomposition residual errors from being entirely discarded. For stable TTS with gently evolving joint subspaces, such as two real-world datasets in Section \ref{subsection_ushcn}, $\beta$ can be broadly chosen from 0 to a positive number close to 1 since the threshold is inactive. For unstable TTS with fast-evolving joint subspaces, such as the wireless channel discussed in Section \ref{subsection_synthetic}, $\beta$ should be set properly to deal with fast-fading or noisy channel states. In addition, we specifically discuss the effects of choice of $\alpha$ and $\varphi$ in Section \ref{subsection_ushcn}.
\end{remark}

Since the structure of \eqref{optim_aaw} is similar to \eqref{optim_updating} except for the sliding time window and AAW, we solve \eqref{optim_aaw} in a similar online manner as Algorithm \ref{algorithm_main}. With AAW, the online updating step in Algorithm \ref{algorithm_main} is modified as follows, while the prediction step is the same.

\textbf{Estimate New Core Tensor} (line 4 in Algorithm \ref{algorithm_main}):
\begin{equation*}
{\bmg}_{T+1}\upt=\frac{1}{1+\varphi\omega_{T+1}\uptone}\left(f_{\setp\upt}(\bmgt\upt)+
\varphi\omega_{T+1}\uptone\bmx_{T+1}\prod_{m=1}^M\times_m(\bmu_m\upt)^H\right).
\end{equation*}

\textbf{Update Projection Matrices} (line 8$-$11 in Algorithm \ref{algorithm_main}): In $(k+1)$-th iteration, ~{for}~$m=1$~{to}~$M$,
$${\bmu}_{m,k+1}\uptone=\bm{L}_{m,k+1}(\bm{R}_{m,k+1})^H,$$
where $\bm{L}_{m,k+1}$ and $\bm{R}_{m,k+1}$ are the left and right singular matrices of
\begin{equation*}
\sum_{t=T-\tau +2}^{T+1}\omega_t\uptone\Big(\bmx_{t}\prod\limits_{i<m}\times_i({\bmu}_{i,k+1}\uptone)^H\prod\limits_{j>m}\times_j(\bmu_{j,k}\uptone)^H\Big)_{(m)} \times\left(\bmg_{t,k}\uptone\right)_{(m)}^H+\frac{\lambda}{2\varphi}\bmu_{m,k}\uptone,
\end{equation*}
respectively.

\textbf{Update core tensor series} (line 12$-$17 in Algorithm \ref{algorithm_main}):  In $(k+1)$-th iteration, ~{for}~$t=T-\tau+2$~{to}~$T+1$,
\begin{equation*}
\bmg_{t,k+1}\uptone = \frac{1}{1+\varphi\omega_t\uptone+\frac{\lambda}{2}}\Big[f_{\setp\uptone_{k+1}}(\setg_{[t-1],k+1}\uptone) +\varphi\omega_t\uptone\bmx_{t}\prod_{m=1}^M\times_m({\bmu}_{m,k+1}\uptone)^H+\frac{\lambda}{2}\bmg_{t,k}\uptone\Big].
\end{equation*}

With the time sliding window, TOPA-AAW further reduces the computational complexity to $\tau/T$ times that of TOPA. In terms of Table \ref{comlexity_ostp}, the total computational complexity of each iteration in TOPA-AAW is
 $O\left(M^2\tau\Bar{I}^M\Bar{R}+M\tau\Bar{R}^M\Bar{I}\right)$.

 \section{Numerical Experiments}\label{section_experiment}

In this section, we evaluate the performance of TOPA in various scenarios, including synthetic low-rank TTS, wireless channel simulations, and two real-world datasets. All numerical experiments are conducted using MATLAB R2018b on a Windows PC with a quad-core Intel(R) Core(TM) 2.0GHz CPU and 8 GB RAM.

To the best of our knowledge, TOPA is the first method for online prediction of generic streaming TTS. In order to provide a reference for comparing the performance of TOPA, we also conduct five other non-online prediction methods:
\begin{itemize}
    \item \textbf{BHT-ARIMA} \citep{bhtarima} and \textbf{MCAR} \citep{moar}: These are two effective offline one-shot prediction methods for TTS. We run these offline methods using the same online observations at each sampling time as the online methods by building their predictors from scratch with multiple iterations to forecast the next tensor data.  BHT-ARIMA employs BHT, an Hankelized tensor expansion technique, to further exploit the temporal relationships among TTS and uses a structure similar to MCAR to formulate predictors. The offline methods perform sufficient iterations until convergence in each prediction, thereby revealing good prediction performance.
    \item \textbf{TOPA-init}:
To illustrate the importance of the online updating for predictors, we test the performance of TOPA without model updating in some experiments. In other words, we use the initial predictor obtained from Algorithm \ref{algorithm_init} and the streaming observations to forecast upcoming data at each sampling time.
    \item \textbf{LSTM} \citep{hochreiter1997long} and \textbf{GMRL} \citep{deng2023learning}: These are two neural-network-based prediction methods. Long Short-Term Memory (LSTM) is a classical time-series forecasting method in the deep learning area, which utilizes a type of recurrent neural network (RNN) architecture that is designed to process and retain information over long sequences. GMRL is a latest work based on neural networks equipped with the tensor decomposition framework.
\end{itemize}

   When comparing the time costs of LSTM and GMRL, we did not account for the training time costs of neural networks, though it may cost much time in practice. On the other hand, owing to the predictor inheritance of TOPA, we conduct TOPA and TOPA-AAW with {\bf only one iteration} at each sampling time during the online prediction.

In addition, taking two real-world datasets in Section \ref{subsection_ushcn} as examples, we discuss the effects of parameter choice for TOPA-AAW. We evaluate the
prediction accuracy of algorithms with the  Normalized Root
Mean Square Error (NRMSE) metric.

\subsection{Synthetic Datasets}\label{subsection_synthetic}

In this subsection, we evaluate the performance of TOPA in synthetic low-rank TTS and wireless channel simulations. The numerical results indicate that TOPA and TOPA-AAW have evident advantage on time-consuming and prediction accuracy.

\textbf{Synthetic Low-Rank Streaming TTS: }
We generate a low-rank noisy streaming TTS  with a similar process as \citet[Sec. 6.3]{synthetic_ref}. The low-rank TTS structure follows the joint Tucker decomposition. We first generate the core tensor series $\{\bmg_t\in\mathbb{C}^{4\times 4\times 4}\}_{t=1}^{70}$ with ARIMA$(3,1,0)$ model. Then, we  generate the joint feature subspace matrices $\{\bmu_m\in\mathbb{U}^{20\times 4}\}_{m=1}^3$, by orthogonalizing randomly generated matrices with i.i.d. $\mathcal{N}(0,1)$ entries. The generated noisy TTS are formulated as
\begin{equation}\label{synthetic_tts}    \bmx_t=\bmg_t\prod_{m=1}^M\times_m\bmu_m+\rho\lVert\bmg_t\rVert_F\bme_t\in\mathbb{C}^{20\times 20\times 20},
\end{equation}
where $\rho=0.1$ is the noise parameter and $\bme_t$ is the noise tensor with i.i.d. $\mathcal{N}(0,1)$ entries. Note that we have $\lVert\bmg_t\rVert_F=\lVert\bmg_t\prod_{m=1}^M\times_m\bmu_m\rVert_F$ and $\rho$ represents the ratio of noise in TTS.

The generated TTS with noise is divided into two sets: the first $T_0=20$ data in the training set, and the rest $50$ in the test set. The regularization parameters for TOPA-AAW are set as $\varphi = 10, \alpha = 0.4, \beta = 0.4$. All numerical results are shown in Figure \ref{synthetic_msre} and Table \ref{synthetic_time}.

Figure \ref{synthetic_msre} compares the average NRMSE of different prediction methods with $10000$ Monte Carlo experiments, each independently and randomly generates a TTS as \eqref{synthetic_tts}.  The ``TOPA-init" line in Figure \ref{synthetic_msre} shows that prediction error increases as prediction continues without model online updating due to the noises in TTS. With the online updating manner, TOPA reveals almost the same performance as MCAR, while BHT-ARIMA has a little advantage over these two methods. Two neural-network-based methods do not perform as well as other algorithms, due to the implicit regression patterns of the core tensor series.

From Table \ref{synthetic_time}, it is easy to see that among all these methods, TOPA and TOPA-AAW are much more efficient owing to the one-loop algorithms, while also keep high prediction accuracy similar to the offline methods.

\begin{figure}[!htbp]
    \centering    \includegraphics[width=0.7\textwidth,height=\textheight,keepaspectratio]{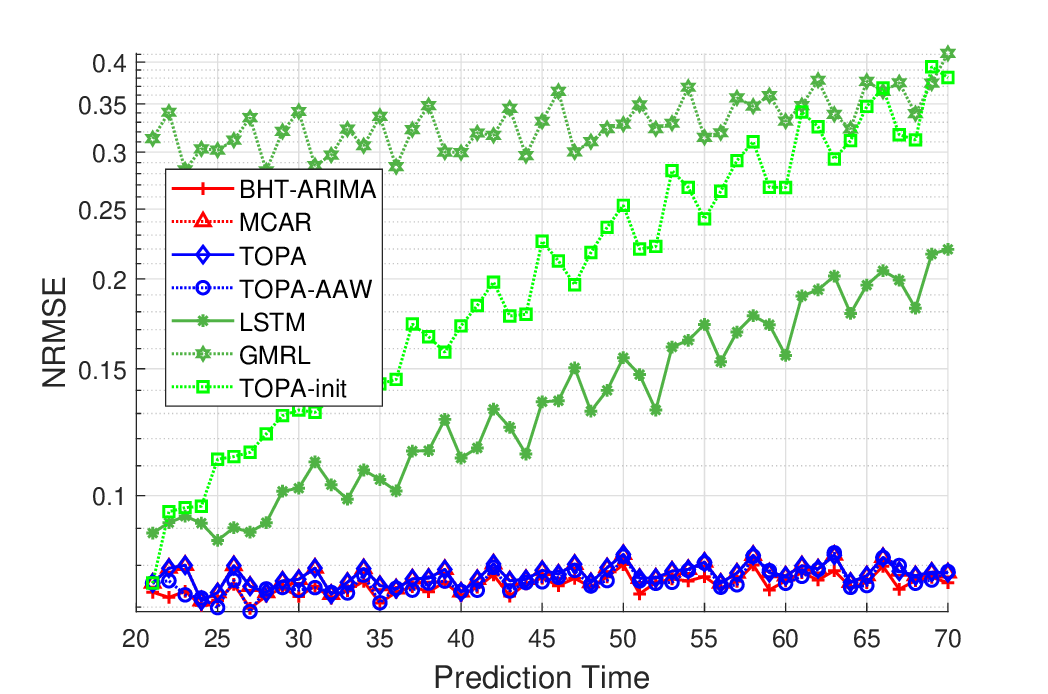}\vspace{-3mm}
    \caption{Performance of different prediction methods on synthetic TTS \eqref{synthetic_tts}. TOPA and TOPA-AAW reveal almost the same performance as the offline methods.}
    \label{synthetic_msre}
\end{figure}

\begin{table}[!h]
\centering
    \caption{Average  time costs and NRMSE of different algorithms on synthetic TTS }\label{synthetic_time}
\resizebox{0.8\textwidth}{!}{\begin{tabular}{ccccccc}
\hline
         & TOPA  & TOPA-AAW       & MCAR  & BHT-ARIMA & GMRL & LSTM  \\ \hline
Time(ms) & \underline{47.8} & \textbf{25.1} & 77.4 & 238.4 & 200.7 & 72.2 \\
NRMSE &  0.0776  &  \underline{0.0760}  &  0.0778  &  \textbf{0.0750}  &  0.3299  &   0.1390  \\
\hline
\end{tabular}}
\end{table}

\textbf{Wireless CSI Prediction:} We consider the wireless channel prediction problem with high-frequency observation and wild fluctuation of statistical features. For fifth-generation wireless communication, many massive multiple-input multiple-output (MIMO) transmission techniques highly rely on accurate CSI acquisition, such as precoding and beamforming. However, the observation of CSI and later data processing usually quickly become outdated due to the short coherent time. CSI prediction is a natural way to  tackle this issue. For time-varying channels, the online manner is important to predict CSI in upcoming time slots.

We use QuaDRiGa \citep{quadriga}  to generate realistic CSI streaming data for single base-station (BS) and single user-equipment (UE) downlink transmission. BS and UE are equipped with $16$ and $2$ antennas, respectively. The transmission occupies the 700MHz frequency with 10MHz bandwidth and 50 subcarriers. During the CSI observation, BS is static, and UE moves with $1.5$m/s speed. The CSI observation frequency is configured as $f_s=25$Hz, so we sample CSI per $40$ms. Then, the CSI observation produces a streaming TTS $\{\bmx_t\in\mathbb{C}^{16\times 2\times 50}\}_{t=1}^\infty$, of which $(\bmx_t)_{ijk}$ denotes the channel state at time $t$ on $k$-th subcarrier between  BS's $i$-th antenna and UE's $j$-th  antenna.

We run $1000$ times Monte Carlo (MC) simulations to evaluate the performance of our proposed algorithms. For each MC simulation, we randomly select a continuous part of CSI TTS with $70$ samples and let the first $T_0=20$ CSI tensors be the training set, while the last $50$ tensors are the streaming data observed in the following $50$ sampling time. The scale of core tensors is set as $(10,2,20)$, which significantly compresses the TTS data. The regularization parameters for the residual of tensor decomposition are set as $\varphi=10, \alpha=0.9, \beta=0.6$. Since the CSI fluctuates wildly, we shorten the time sliding window in TOPA-AAW to $\tau=8$. All prediction methods employ ARIMA$(2,1,1)$ model for building temporal relations among core tensor series. To further improve the adaptability of TOPA/TOPA-AAW for the time-varying property of CSI, we conduct two iterations instead in the online updating step of Algorithm \ref{algorithm_main}. All numerical results are shown in Figure \ref{channel} and Table \ref{timechannel}.

\begin{figure}[ht]
    \centering    \includegraphics[width=0.7\textwidth,height=\textheight, keepaspectratio]{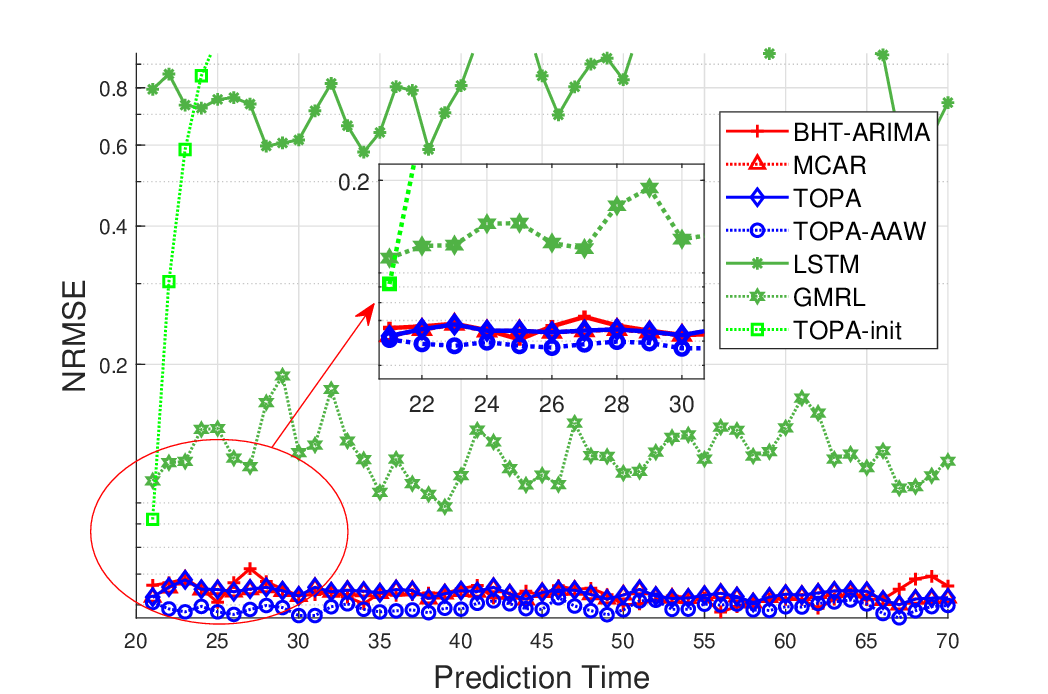}
    \caption{Performance of different prediction methods on CSI prediction. TOPA-AAW has the best performance, while TOPA performs as well as the offline methods.}
    \label{channel}
\end{figure}

\begin{table}[ht]
\centering
    \caption{Average Time costs and NRMSE of different algorithms on CSI prediction}\label{timechannel}
\resizebox{0.85\textwidth}{!}{\begin{tabular}{ccccccc}
\hline

         & TOPA  & TOPA-AAW       & MCAR  & BHT-ARIMA & GMRL & LSTM  \\ \hline
Time(ms) &
     186.8 &  \underline{37.6} & 915.8 &  2634.7 & \textbf{22.6} & 50.7 \\
NRMSE & 0.0636  &  \textbf{0.0591} & 0.0635  &  0.0634        & 0.1295    & 0.9170   \\
     \hline
\end{tabular}}
\end{table}
Figure \ref{channel} illustrates the average NRMSE of CSI online prediction for different prediction methods. The ``TOPA-init" line in Figure \ref{channel} reveals the extreme change in the wireless environment, even in low-speed mobility scenarios. As an online method, TOPA shows very close performance to two offline methods, while the prediction accuracy of TOPA-AAW is much better than TOPA and MCAR, owing to the good adaptation of AAW to the wireless channel evolution.

Table \ref{timechannel} presents the average time costs and NRMSE of six methods in each prediction. It should be emphasised that the training of the GMRL and LSTM methods cost up to a few minutes, which is not considered in Table \ref{timechannel}.  TOPA-AAW is much more efficient than other methods except GMRL,  while has the best prediction performance. Table \ref{timechannel} illustrates that the difference in the number of necessary alternative iterations for building predictors leads to a significant difference in the time costs between online and offline methods.

\subsection{Real-world Datasets: USHCN and NASDAQ100}\label{subsection_ushcn}

In this subsection, We  apply our prediction algorithm to two real-world datasets:
\begin{itemize}
    \item \textit{USHCN}\footnote{This dataset is in https://www.ncei.noaa.gov/pub/data/ushcn/v2.5/}: this dataset records the monthly climate data of 1218 weather stations in the United States during the past 100 years, including four statistical features: monthly mean maximum temperature, monthly mean minimum temperature, monthly minimum temperature, and monthly total precipitation. 120 meteorological stations with relatively complete data are screened for numerical experiments. With their quarterly average observation data from 1940 to 2014, we establish a $75$-length TTS with the scale $120\times 4\times 4$ of each tensor data.
    \item \textit{NASDAQ100}\footnote{This dataset is in https://github.com/Karin-Karlsson/stockdata}: This dataset records the daily business data of 102 NASDAQ-listed companies over 90 days in 2014, including five statistical features: opening quotation, highest quotation, lowest quotation, closing quotation and adjusting the closing price. Therefore, we obtain a 90-length TTS with the scale $102\times 5$ of each matrix data.
\end{itemize}

We divide the screened USHCN TTS into two parts: the first $T_0=40$ data are used for finding the initial predictor and predicting the data at sampling time $41$, and the last $35$ data are regarded as the streaming observation of USHCN TTS. In terms of the discussion in \citet[Sec. V]{moar}, we choose AR(2) model to build the temporal correlations among core tensor series. In the structure of joint tensor decomposition, the scale of core tensors is given as $12\times 4\times 4$.  Since the statistical characteristics of data are nearly stable, we configure a wide time sliding window in TOPA-AAW with length $\tau=20$.

For the NASDAQ100 dataset, we find the initial predictor with the first $T_0=60$ traffic flow data, and observe the last $30$ data stream. The ARIMA$(3,1,0)$ model captures the statistical characteristics in NASDAQ100 TTS. The time sliding window used in TOPA-AAW is configured as $\tau=20$.

We first research the effects of different choices of $\alpha$ and $\varphi$ in \eqref{optim_aaw}. Figure \ref{para_err} shows the average NRMSE of the proposed TOPA-AAW algorithm with different values of parameters $\varphi$ and $\alpha$. Regarding Figure \ref{para_err} (a), we can see that the best $\varphi$ of USHCN and NASDAQ100 datasets are  0.2 and 20, respectively. For both two datasets, the prediction performance is unsatisfactory when $\varphi$ is relatively small, which manifests the necessity of regularization. For stable TTS, small $\varphi$ makes the regularization term close to $0$ with sufficiently accurate joint Tucker decomposition, hence the objective function in \eqref{optim_aaw} is dominated by the first term, which leads to inaccurate joint decomposition with online prediction continuing. Furthermore, when $\varphi$ is large enough, the accuracy of TOPA-AAW is improved to a stable level. From  Figure \ref{para_err} (b), the best  $\alpha$ for USHCN and NASDAQ100 datasets is close to $1$, which illustrates the effectiveness of reducing the weight of stale data. The prediction performance is extremely poor when $\alpha$ is set as $1$, since the regularization in \eqref{optim_aaw} with $\alpha=1$ only retains the last term with time $t=T+1$. In brief, when $\alpha<1$ and $\varphi$ are large enough, the accuracy of TOPA-AAW is not very sensitive to the choice of these parameters, which allows us to choose the parameter more flexibly. In the following experiments, we configure the regularization parameters for TOPA-AAW as: $\varphi=0.5$ and $\alpha=0.98$ for USHCN, and $\varphi=20$ and $\alpha=0.99$ for NASDAQ100. Moreover, as discussed in Remark \ref{remark_beta}, we set $\beta=0.5$.

\begin{figure}[!ht]
	\centering
\includegraphics[width=0.7\textwidth,height=\textheight,keepaspectratio]{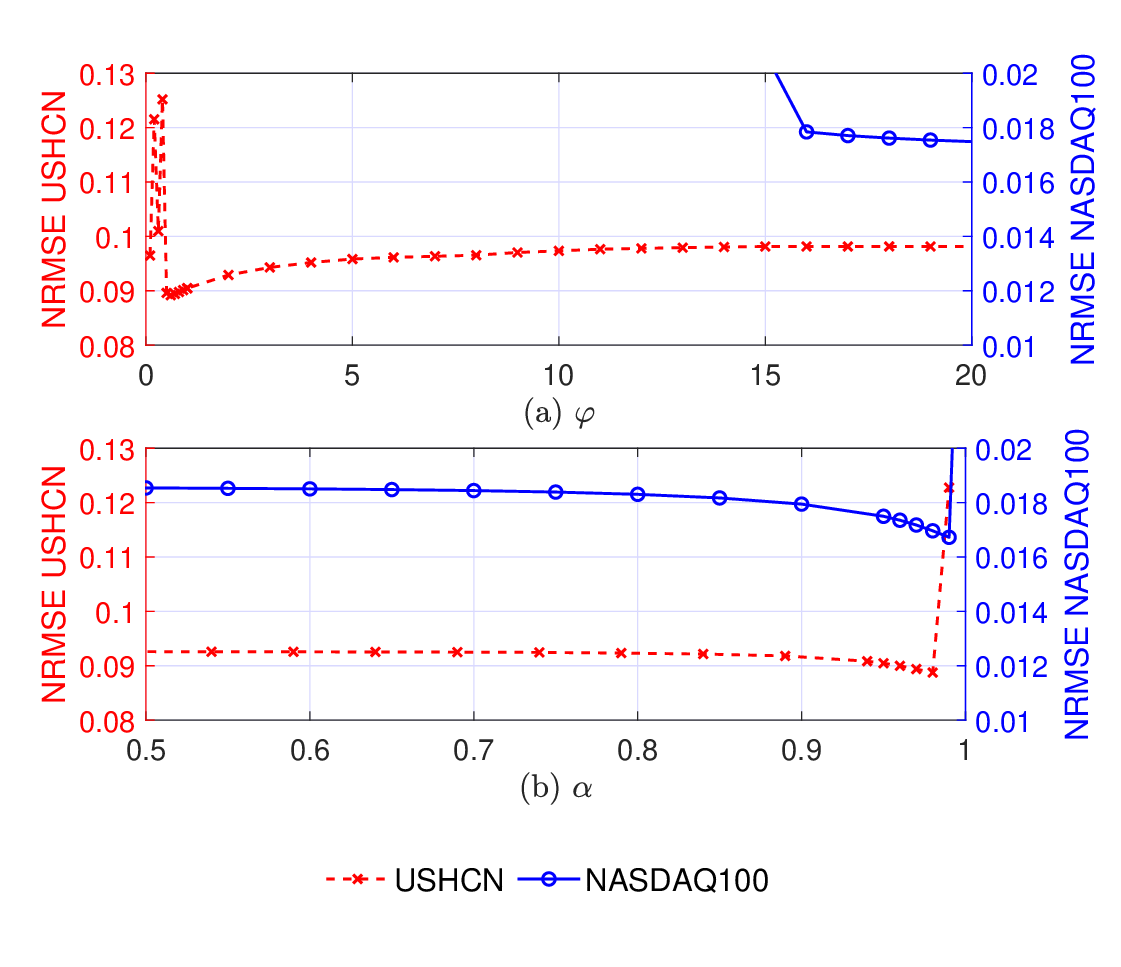}\vspace{-4mm}
	\caption{Effect of parameter $\varphi$ and $\alpha$ for the TOPA-AAW algorithm on two real-world datasets: (a) average NRMSE with $\alpha=0.95$ and different $\varphi$, and (b) average NRMSE with $\varphi_{\text{UHSCN}}=0.5,$ $\varphi_{\text{NASDAQ100}}=20$ and different $\alpha$. When $\alpha<1$ and $\varphi$ are large enough, the accuracy of TOPA-AAW is not sensitive to the choice of these parameters. }
	\label{para_err}
\end{figure}

Figure \ref{core_size} and Table \ref{time_coresize_USHCN} show the average NRMSE and time costs of the proposed TOPA-AAW algorithm with different scales of core tensors, respectively. We implement the experiments with three different sizes of the core tensors: large scale with dimensions of $80\times 4\times 4$, medium scale with dimensions of $12\times 4\times 4$, and small scale with dimensions of $3\times 3\times 3$. The results show that a larger scale of core tensors leads to higher prediction accuracy with more time costs. Therefore, choosing the right size of core tensor is a trade-off between prediction accuracy and time costs.

\begin{figure}[!ht]
    \centering    \includegraphics[width=0.7\textwidth,height=\textheight, keepaspectratio]{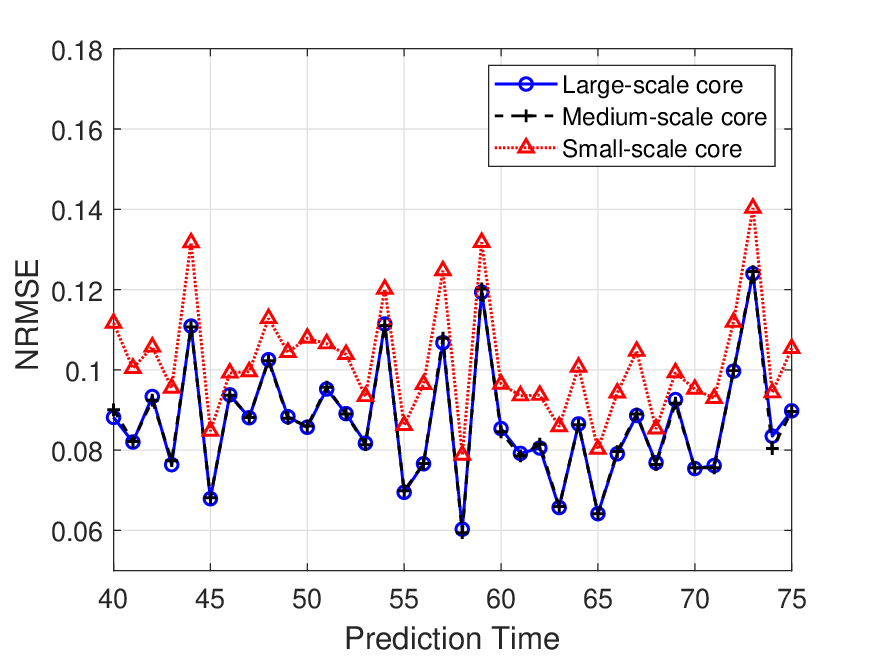}
    \caption{Performance of different sizes of core tensors for TOPA-AAW algorithm on USHCN dataset.}
    \label{core_size}
\end{figure}
\begin{table}[!h]
\centering
    \caption{Average time costs of different sizes of core tensors for TOPA-AAW algorithm on USHCN dataset}\label{time_coresize_USHCN}
\resizebox{0.65\textwidth}{!}{\begin{tabular}{cccc}
\hline
         & Large scale  &  Medium scale       & Small scale    \\ \hline
Time(ms) & 8.7 & 5.1 & 4.8 \\ \hline
\end{tabular}}
\end{table}

Next we focus on the impact of the length of the time sliding window on the prediction performance of TOPA-AAW. As shown in Figure \ref{Time_window}, a longer sliding time window  (with $\tau = 5$ and $\tau = 10$) leads to better prediction performance. On the other side, the complexity analysis in Section \ref{subsection_complexity} proposes that the computational complexity of TOPA-AAW is positively proportional to the time window length $\tau$. With the trade-off between prediction accuracy and time costs, we choose the length of time sliding window $\tau = 20$ in the following experiments.

\begin{figure}[!ht]
    \centering
    \includegraphics[width=0.65\textwidth,height=\textheight, keepaspectratio]{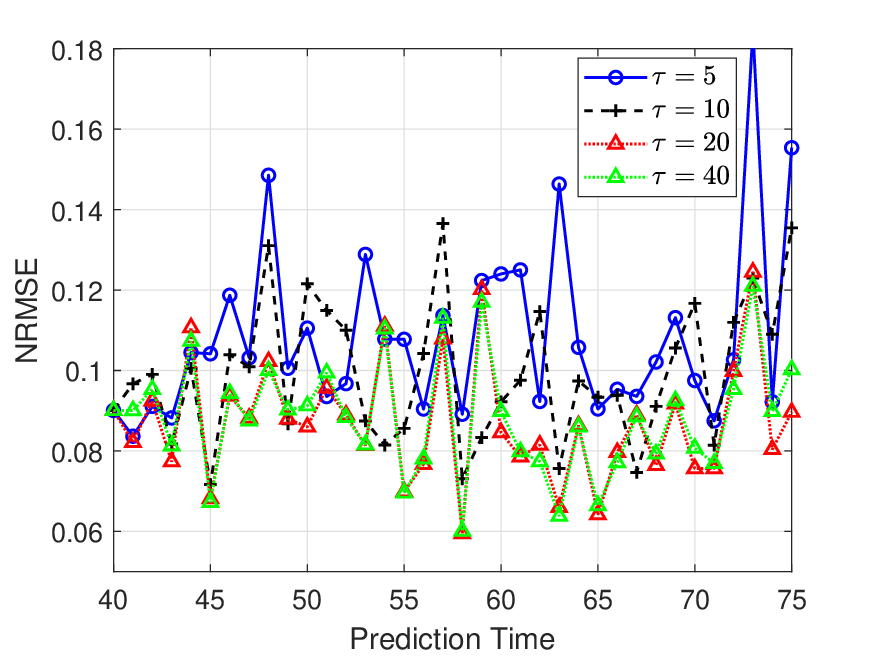}\vspace{-4mm}
    \caption{Performance of different sliding time window lengths for TOPA-AAW algorithm on USHCN dataset.}
    \label{Time_window}
\end{figure}

Figure \ref{ushcn} depicts the comparison results of prediction performance. TOPA reveals nearly the same accuracy for two real-world datasets as offline methods, though it runs much fewer iterations than MCAR in each prediction. TOPA-AAW can further improve the accuracy of TOPA by tracking the latest statistical patterns in TTS. In the NASDAQ100 dataset, the prediction accuracy of TOPA-AAW is nearly $20\%$ better than TOPA and MCAR, while three methods show significant advantages over BHT-ARIMA and two neural-network-based methods.

\begin{figure}[!ht]
    \centering    \includegraphics[width=0.65\textwidth,height=\textheight,keepaspectratio]{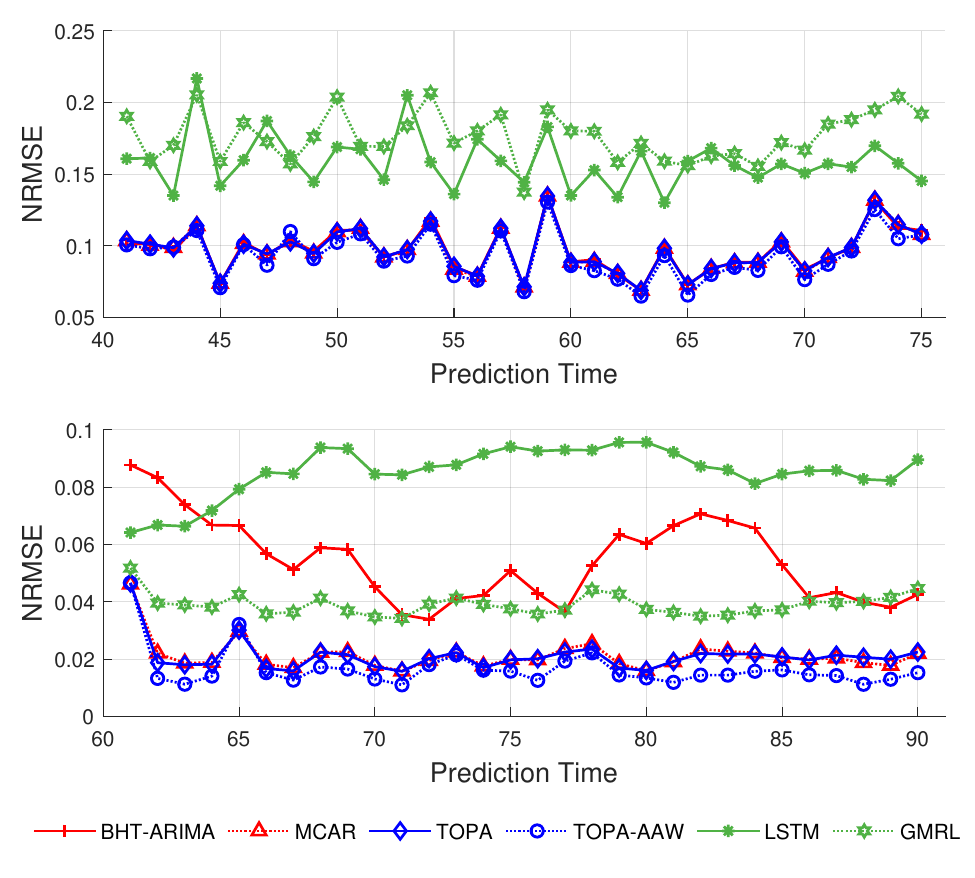}\vspace{-4mm}
    \caption{Performance of different prediction methods on two real-world datasets: (a) USHCN, and (b) NASDAQ100. TOPA-AAW shows the best performance owing to AAW.}
    \label{ushcn}
\end{figure}

Table \ref{timeUSHCN} presents the average time costs of six prediction methods. Owing to the online manner, TOPA/TOPA-AAW can provide prediction results efficiently and accurately.
\begin{table}[!ht]
\centering
    \caption{Average time costs of different algorithms on USHCN and NASDAQ100 datasets}\label{timeUSHCN}
\resizebox{0.9\textwidth}{!}{\begin{tabular}{ccccccc}
\hline

Time(ms)         & TOPA  & TOPA-AAW       & MCAR  & BHT-ARIMA & GMRL & LSTM  \\ \hline
USHCN & 59.2 & \textbf{5.1} & 328.7 & 1583.1 & {15.0} & 34.6 \\ \hline
NASDAQ100 & 53.5 & \underline{11.4} & 263.0 & 419.2 & \textbf{6.9} & 26.2 \\ \hline
\end{tabular}}
\end{table}

\section{Conclusion}\label{sectioncon_clusion}

In this paper, we present a joint-tensor-factorization-based online prediction algorithm for streaming tensor time series. By leveraging tensor factorization, our algorithm effectively compresses the streaming data while capturing the underlying intrinsic correlations. The proposed online updating scheme significantly enhances the speed of predictor updating and can maintain high prediction accuracy. We also analyze the convergence of the proposed algorithm. Additionally, we introduce automatically adaptive weights to address the challenge of data staleness in streaming data. Through the numerical experiments in various scenarios, we observe promising results that validate the effectiveness and efficiency of our approach.

\section*{Acknowledgements}
{The author Liping Zhang  was supported by the National Natural Science Foundation of China (Grant No. 12171271).}

\vskip 0.2in
\bibliography{Efficient_Online_Prediction_for_High-Dimensional_Time_Series_via_Joint_Tensor_Tucker_Decomposition}

\end{document}